\documentclass[11pt, a4paper,reqno]{amsart}

\usepackage{amsmath}
\usepackage{amssymb}
\usepackage{amsthm}
\usepackage[shortlabels]{enumitem}
\usepackage[british]{babel}
\usepackage{hyperref}

\calclayout
\hypersetup{colorlinks=true,linkcolor=blue,citecolor=blue,urlcolor=blue}

\theoremstyle{definition}
\newtheorem{theorem}{Theorem}[section]
\newtheorem{remark}[theorem]{Remark}
\newtheorem{lemma}[theorem]{Lemma}
\newtheorem{coro}[theorem]{Corollary}
\numberwithin{equation}{section}

\newcommand{\abs}[1]{\left\lvert#1\right\rvert}
\newcommand{\norm}[1]{\left\|#1\right\|}
\newcommand{\R}{\mathbb R}
\newcommand{\N}{\mathbb N}
\newcommand{\Z}{\mathbb Z}
\renewcommand{\epsilon}{\varepsilon}
\newcommand{\A}{\mathcal A}
\newcommand{\dd}{\, \mathrm d}
\newcommand{\id}{\mathrm{Id}}
\newcommand{\cF}{\mathcal F}
\newcommand{\D}{\mathcal D}
\newcommand{\m}{\mathbf m}
\newcommand{\homd}{\mathsf Q}
\newcommand{\cW}{\mathcal W}
\newcommand{\E}{\mathcal E}
\newcommand{\cS}{\mathcal S}
\DeclareMathOperator{\supp}{supp}
\renewcommand{\L}{\operatorname{L}}
\newcommand{\C}{\operatorname{C}}
\renewcommand{\H}{\operatorname{H}}
\newcommand{\W}{\operatorname{W}}
\DeclareRobustCommand{\Hdot}{\dot{\H}\protect{\vphantom{H}}}
\DeclareRobustCommand{\Wdot}{\dot{\W}\protect{\vphantom{W}}}
\newcommand{\kforc}{\vec{K}}
\newcommand{\kpi}{\vec{K}^{\mathrm{pi}}}

\title[Transfer of regularity by kinetic mollification]{Transfer of regularity by kinetic mollification along critical trajectories}

\author{Lukas Niebel}
\address[Lukas Niebel]{Institut f\"ur Analysis und Numerik, Universit\"at M\"unster\\
Orl\'eans-Ring 10, 48149 M\"unster, Germany.}
\email{lukas.niebel@uni-muenster.de}
\date{\today}



\begin{document}
\allowdisplaybreaks

\begin{abstract}
	We give a trajectory-based proof of transfer-of-regularity estimates \`a la Bouchut--H\"ormander for kinetic equations at the weak scale of local diffusion. The proof is based on kinetic mollification along critical kinetic trajectories, whose endpoint map has precisely the kinetic scaling. Instead of using Fourier multipliers or the Kolmogorov fundamental solution, we work directly in physical space. The key step is an explicit formula for the mollification defect and corresponding critical support, size, and difference estimates for the kernels. These bounds yield the sharp, scale-invariant homogeneous estimates through maximal-function and Littlewood--Paley estimates on the kinetic group.
\end{abstract}

\maketitle

\section{Introduction}

A central theme in the theory of kinetic partial differential equations is the transfer of regularity from the diffusive variable $v$ to the spatial variable $x$. Suppose that $f=f(t,x,v)\colon \R^{1+2d}\to\R$ enjoys some regularity in $v$, for instance $\nabla_v f\in X$, and that the kinetic transport term is controlled in the sense that
\begin{equation*}
	(\partial_t + v\cdot \nabla_x)f \in Y
\end{equation*}
for suitable function spaces $X$ and $Y$. The question is then how much regularity in the $x$-variable can be inferred from these assumptions. A basic example is the weak-solution setting, where $X=\L^2_{t,x,v}$ and $Y=\L^2_{t,x}\Hdot^{-1}_v$, that is,
\begin{equation*}
	(\partial_t + v\cdot \nabla_x)f = \nabla_v \cdot S_0
\end{equation*}
for some $S_0\in \L^2_{t,x,v}$.

Since the seminal work of H\"ormander \cite{hormander_hypoelliptic_1967}, it has been known that such assumptions imply fractional Sobolev regularity in the $x$-variable. His approach is geometric: regularity in the $\nabla_x$ direction is recovered by commuting the vector fields $\partial_t + v\cdot \nabla_x$ and $\nabla_v$, using the identity
\begin{equation*}
	[\nabla_v,\partial_t + v\cdot \nabla_x] = \nabla_x.
\end{equation*}
This yields a fractional Sobolev gain with a small, implicit exponent $s>0$.

A major advance was made by Bouchut \cite{bouchut_hypoelliptic_2002}, who proved transfer-of-regularity estimates in $\L^p$ and, for the first time, obtained the optimal exponent $s=\frac13$. More precisely, if $f\in \L^p_{t,x,v}$, $\nabla_v f\in \L^p_{t,x,v}$, and
\begin{equation*}
	(\partial_t + v\cdot \nabla_x)f = \nabla_v \cdot S_0
\end{equation*}
for some $S_0\in \L^p_{t,x,v}$, then
\begin{equation*}
	\norm{D_x^{1/3}f}_{\L^p_{t,x,v}} \lesssim \norm{\nabla_v f}_{\L^p_{t,x,v}}^{\frac23}\norm{S_0}_{\L^p_{t,x,v}}^{\frac13},
\end{equation*}
where $D_x^s := (-\Delta_x)^{\frac{s}{2}}$ for any $s>0$. A scaling argument shows that, for global estimates, $\frac13$ is the correct gain of regularity in $x$, and that the estimate is sharp. Bouchut's proof relies on Fourier analysis and multiplier estimates. Another way to quantify this phenomenon is through velocity averaging lemmas \cite{MR0808622,MR0923047,MR1932972}; in fact, Bouchut's theorem implies optimal estimates for velocity averages.

Bouchut's result was later reproved in the Hilbertian case $p=2$ in \cite{MR2948893} by working with characteristics at the Fourier level. The same transfer of regularity can also be derived from the fundamental solution of the Kolmogorov equation
\begin{equation*}
	(\partial_t + v\cdot \nabla_x)f - \Delta_v f = \nabla_v \cdot \tilde{S}_0,
\end{equation*}
see \cite{niebel_kinetic_2021,AIN,an_kfpLp_2025}. The idea is to first prove an estimate on solutions to the Kolmogorov equation with divergence source term and then apply this estimate to the source term $\tilde{S}_0 = -\nabla_v f +S_0$. Bouchut's theorem was recovered and extended in \cite{AIN,an_kfpLp_2025} via this strategy.
A related argument based on the fundamental solution was used in \cite{guerand_quantitative_2022} to obtain transfer of regularity for weak subsolutions.

Moreover, H\"ormander's commutator method yields the critical exponent $s=\frac13$ in a weaker Besov-type norm; see \cite{MR4413303,albritton2021variational} and also \cite[Section 4]{bouchut_hypoelliptic_2002}. For strong solutions, one obtains an analogous gain under the assumptions $(\partial_t + v \cdot \nabla_x)f \in \L^p_{t,x,v}$ and $\Delta_v f\in \L^p_{t,x,v}$, namely $D_x^{\frac23}f\in \L^p_{t,x,v}$; see \cite{bouchut_hypoelliptic_2002,MR4336467,hirao2026anisotropicmaximallpregularityestimates,MR4444079}.

Analogous statements also hold for fractional regularity in $v$. Bouchut \cite{bouchut_hypoelliptic_2002} proved in the strong-solution setting that if $(\partial_t + v\cdot \nabla_x)f\in \L^p_{t,x,v}$ and $D_v^\beta f\in \L^p_{t,x,v}$, then $D_x^{\frac{2\beta}{2\beta+1}}f\in \L^p_{t,x,v}$. For $p=2$, this was extended in \cite{AIN,niebel_kinetic_2021} to a broad scale of kinetic spaces, including the weak-solution setting. More recently, \cite{an_kfpLp_2025} established transfer-of-regularity estimates in $\L^p$-based kinetic Sobolev spaces at every regularity scale, again including weak $\L^p$ solutions with fractional regularity in $v$.

\medskip

In the present work we develop a trajectory-based approach to transfer-of-regularity based on kinetic mollification along critical kinetic trajectories, following \cite{dietert2025criticaltrajectorieskineticgeometry}; see also \cite{niebel2023kinetic,anceschi2024poincare,dietert2025nashsgboundkolmogorov}. Our method can be viewed as a quantitative, physical-space, commutator-free realisation of H\"ormander's original commutator argument. It avoids explicit computations in Fourier variables and does not rely on the fundamental solution, while bypassing H\"ormander's detour through the interpolation space corresponding to $(\partial_t + v\cdot \nabla_x)^{\frac12}f\in \L^2_{t,x,v}$; compare \cite{hormander_hypoelliptic_1967,MR4413303,albritton2021variational}.

The contribution is methodological.
The sharp Sobolev estimate stated in Theorem \ref{thm:kintorLpGeneral} is known in several forms, beginning with Bouchut's Fourier-multiplier proof \cite{bouchut_hypoelliptic_2002} and, in a broader kinetic Sobolev framework, through the Kolmogorov-equation approach of \cite{AIN,an_kfpLp_2025}. The point of the present note is not to claim a new inequality, but to give a direct proof mechanism which works at the level of the kinetic geometry itself.

The difficult part is to realise the formal commutator identity
\[
	[\nabla_v,\partial_t+v\cdot\nabla_x]=\nabla_x
\]
by an actual averaging operator at the critical scale. A generic averaging along characteristics does not give the right spatial gain. One needs kinetic trajectories whose endpoint map is non-degenerate at the anisotropic kinetic scale
\[
	t\sim r^2,\qquad v\sim r,\qquad x\sim r^3,
\]
and whose inverse has exactly the losses $r^{-3}$ in the spatial variable and $r^{-1}$ in the velocity variable, i.e.\ works at the scale of weak solutions; see \textbf{(M2)}--\textbf{(M3)} below. This is the role of the oscillatory forcings $g_1,g_2$. They encode the missing spatial direction in physical space and produce an approximate identity $T_{K_\tau}$ whose defect admits the exact representation \eqref{eq:repfTtauf}.

This representation is the main technical device of the paper. Once it is available, the exponent $\frac13$ appears from the elementary scale considerations of the kernels rather than from a Fourier multiplier calculation or from estimates on the Kolmogorov fundamental solution. This provides a reusable physical-space mechanism to prove transfer of regularity.     This appears to be the first proof of Bouchut's theorem that does not rely on Fourier calculations or the Kolmogorov fundamental solution.

The present paper is organised as follows. In Section \ref{sec:kintraj} we recall critical kinetic trajectories and the associated kinetic mollification. In Section \ref{sec:torHB} we prove transfer-of-regularity estimates in H\"ormander's Besov spaces, and in Section \ref{sec:torLp} we establish their $\L^p$-based Sobolev counterparts.

For simplicity, we prove all estimates for smooth functions and do not discuss the extension to general elements of kinetic Sobolev spaces; for the relevant technical arguments, see \cite{silvestre_boundary_2022,an_kfpLp_2025,dietert2025criticaltrajectorieskineticgeometry}. Throughout, $d\in \N$ denotes the dimension. The subscripts $t$, $x$, and $v$ refer to the time variable, the spatial variables, and the velocity variables, respectively. We use these subscripts consistently for functions, gradients, Lebesgue spaces $\L^p$, and Sobolev spaces $\H^s$, whenever the meaning is clear from the context. Finally, the symbols $\lesssim$, $\gtrsim$, and $\approx$ denote comparisons up to universal constants. In proofs, $C$ denotes a universal constant that may change from line to line.

\textbf{Acknowledgements.}
This work was motivated by helpful correspondence with Fran\c{c}ois Bouchut concerning the problem of proving transfer of regularity without relying on Fourier-transform methods. The author also thanks Pascal Auscher, Helge Dietert, Cyril Imbert, Cl\'ement Mouhot, and Rico Zacher for valuable discussions.
This work is supported by the Deutsche Forschungsgemeinschaft (DFG, German Research Foundation) under Germany's Excellence Strategy EXC 2044/2--390685587, Mathematics M\"unster: Dynamics--Geometry--Structure.

\section{Kinetic trajectories and kinetic mollification} \label{sec:kintraj}
Our approach is based on mollification in the direction of critical kinetic trajectories as proposed in \cite{dietert2025criticaltrajectorieskineticgeometry}. We recall them for the convenience of the reader.
The adjective ``critical'' refers to the fact that the trajectories below are tuned to the exact hypoelliptic scaling relevant for the gain $D_x^{1/3}$. The construction is not merely a convenient parametrisation of characteristics.

All $2 \times 2$ matrices act on $\R^{2d}$ as their tensor with $\id_d$ and $(M)_{i;j}$ denotes the $(i,j)$-block (a scalar multiple of \(\id_d\)) of a $2\times2$ block matrix.

\subsection{Critical kinetic trajectories}
We consider the two forcings $g_1,g_2 \colon [0,\infty) \to \R$ defined by
\begin{align*}
	g_1(r) =  r^{3} \sin \log r
	\quad\text{ and }\quad
	g_2(r) =  r^{3} \cos \log r
\end{align*}
for $r>0$ and continuously extended by $0$ at $r = 0$.

Given $(t,x,v) \in \R^{1+2d}$, $m_0 \in \R\setminus\{0\}$ and $(m_1,m_2) \in \R^{2d}$ we define
\[
	\gamma^{\mathbf m} = (\gamma^{\mathbf m}_t,\gamma^{\mathbf m}_x,\gamma^{\mathbf m}_v) \colon [0,\infty) \times \R^{1+2d}\times \R^{1+2d} \to \R^{1+2d}
\]
with $\mathbf m = (m_0,m_1,m_2)$ as
\begin{equation} \label{eq:gammam}
	\gamma^{\mathbf m}(r;(t,x,v))
	:= \begin{pmatrix}
		t+m_0 r^2 \\
		\E_{m_0}(r)\begin{pmatrix}
			           x \\
			           v
		           \end{pmatrix} + \D_{m_0} \cW(r)\D_{m_0}^{-1}
		\begin{pmatrix}
			m_1 \\  m_2
		\end{pmatrix}
	\end{pmatrix}
\end{equation}
where we have denoted, for $\delta \in \R$ and $r \in [0,\infty)$,
\begin{equation} \label{eq:defDandE}
	\D_{\delta} := \begin{pmatrix}
		\delta \, \id_d & 0     \\
		0               & \id_d
	\end{pmatrix}, \quad
	\E_\delta(r) = \begin{pmatrix}
		\id_d & \delta r^2 \, \id_d \\
		0     & \id_d
	\end{pmatrix} \mbox{ and } \cW(r) = \begin{pmatrix}
		g_1(r)                  & g_2(r)                  \\
		\frac{\dot{g}_1(r)}{2r} & \frac{\dot{g}_2(r)}{2r}
	\end{pmatrix}.
\end{equation}
We interpret the entries $\frac{\dot g_i(r)}{2r}$ at $r=0$ by continuous extension. With this convention the curves $r\mapsto \gamma^{\m}(r;(t,x,v))$ are absolutely continuous on compact intervals, although the entries $\frac{\dot g_i(r)}{2r}$ need not be classically differentiable at $r=0$. All derivatives in $r$ below are therefore understood for $r>0$.

The mapping $\gamma^{\mathbf m}$ satisfies the following properties.
\begin{enumerate}[itemsep=0.2cm]
	\item[\hypertarget{link:M1}{\textbf{(M1)}}] It is an  open-ended kinetic trajectory:  $\dot \gamma_x ^{\mathbf m} = \dot \gamma_t ^{\mathbf m} \gamma_v ^{\mathbf m} $ (note that $\dot \gamma_t^{\mathbf m} = 2m_0r$).
	\item[\hypertarget{link:M2}{\textbf{(M2)}}] We have $ \det(\D_{m_0} \cW(r)\D_{m_0}^{-1}) = (-2)^{-d} r^{4d}$ for all $r \in [0,\infty)$.
	\item[\hypertarget{link:M3}{\textbf{(M3)}}] We have
	      \[
		      \abs{((\D_{m_0} \cW(r)\D_{m_0}^{-1})^{-1})_{i;1}} \lesssim (1+\abs{{m_0}}^{-1}) r^{-3}
	      \]
	      and
	      \[
		      \abs{((\D_{m_0} \cW(r)\D_{m_0}^{-1})^{-1})_{i;2}} \lesssim (1+\abs{{m_0}}) r^{-1}
	      \]
	      for $i=1,2$ and $r \in (0,\infty)$.
	\item[\hypertarget{link:M4}{\textbf{(M4)}}] The following bounds hold:
	      \begin{equation*}
		      \begin{cases}
			      \abs{\dot{\gamma}_v^{\mathbf m}(r)} \lesssim \left( \frac{\abs{m_1}}{\abs{m_0}} + \abs{m_2} \right), \\
			      \abs{\gamma_v^{\mathbf m}(r)-v} \lesssim \left( \frac{\abs{m_1}}{\abs{m_0}} + \abs{m_2} \right)r,    \\
			      \abs{\gamma_x^{\mathbf m}(r)-x-m_0vr^2} \lesssim \left( {\abs{m_1}} + \abs{m_0}\abs{m_2} \right)r^3
		      \end{cases}
	      \end{equation*}
	      for all $r \in (0,\infty)$.
\end{enumerate}

We introduce $\A_{m_0}(r) := \D_{m_0} \cW(r)\D_{m_0}^{-1}$.
To simplify notation in later proofs, for $r>0$ we introduce the forcing matrix as
\begin{equation*}
	\cF_{m_0}(r) = \begin{pmatrix}
		\frac{1}{m_0}\partial_r \frac{\dot{g}_1(r)}{2r} & \partial_r\frac{\dot{g}_2(r)}{2r}
	\end{pmatrix}
\end{equation*}
which allows us to write $\dot{\gamma}_v^{\m}(r;(t,x,v)) = \cF_{m_0}(r)\binom{m_1}{m_2}$ for $r>0$.

\begin{remark}
	These trajectories can be constructed as in \cite[Section 2.1]{dietert2025criticaltrajectorieskineticgeometry}. However, in comparison to \cite{dietert2025criticaltrajectorieskineticgeometry} we consider time in a quadratic scaling $r^2$ instead of linear scaling $r$. Conceptually these choices are equivalent. Moreover, in view of \cite[Section 2.1]{dietert2025criticaltrajectorieskineticgeometry} it is clear how to construct kinetic trajectories that connect any two given points with this scaling.
\end{remark}

\subsection{Phase-space kinetic mollification}

To simplify the following calculations we introduce some notation. We recall that for any $(t,x,v),(s,y,w) \in \R^{1+2d}$ the kinetic translation group is defined via
\begin{equation*}
	(t,x,v) \circ (s,y,w) = (t+s,x+y+sv,v+w)
\end{equation*}
and inverses are given as $(t,x,v)^{-1} = (-t,-x+tv,-v)$. We write $\homd = 4d+2$ for the homogeneous dimension.

Let $N \in \N$. To any kernel $J \colon \R^{1+2d} \to \R^N$ we associate the integral operator along kinetic translation as follows
\begin{equation} \label{eq:defTJ}
	[T_J f](t,x,v) = \int_{\R^{1+2d}} J\left( (t,x,v)^{-1} \circ \m \right) \cdot f(\m) \dd \m
\end{equation}
for $(t,x,v) \in \R^{1+2d}$ and $f \colon \R^{1+2d} \to \R^N$ and where we write $\m = (m_0,m_1,m_2) \in \R^{1+2d}$. For convenience we provide $(t,x,v)^{-1} \circ \m = (m_0-t,m_1-x-(m_0-t)v,m_2-v)$ and note that
\begin{equation} \label{eq:Tjalternative}
	[T_J(f)](t,x,v) = \int_{\R^{1+2d}} J(s,y,w)\cdot f((t,x,v)\circ(s,y,w)) \dd (s,y,w).
\end{equation}
In the literature this is also known as kinetic convolution and often denoted by $\ast_{\rm kin}$.
Here and below, if $f$ and $J$ take values in $\R^N$, then $J \cdot f$ denotes the Euclidean inner product in $\R^N$. For $N=1$ this is ordinary multiplication.

We consider any nonnegative function $\psi \in \C_c^\infty(\R^{1+2d})$ with unit mass and support in
\begin{equation*}
	\supp \psi \subset (-2,-1) \times B_1(0) \times B_1(0).
\end{equation*}
Consequently, whenever a kernel below contains an expression of the form
\[
	\psi\left(\frac{s}{r^2},\A_{\frac{s}{r^2}}(r)^{-1}\binom{y}{w}\right),
\]
we understand it to be zero if $s/r^2\notin(-2,-1)$. Thus $\A_{s/r^2}(r)^{-1}$ is only evaluated for nonzero $s/r^2$.

The idea of kinetic mollification is to consider the averaging of a function $f = f(t,x,v)$ in the direction of a kinetic trajectory with parameter $\m$ at artificial time $\tau$:
\begin{equation*}
	(t,x,v) \mapsto \int_{\R^{1+2d}} f(\gamma^{\m}(\tau;(t,x,v))) \psi(\m) \dd \m.
\end{equation*}

By a change of variables we may rewrite this as an integral operator along kinetic translation:
\begin{equation*}
	T_{K_\tau}(f) = \int_{\R^{1+2d}} f(\m) K_\tau\left( (t,x,v)^{-1} \circ \m \right) \dd \m = \int_{\R^{1+2d}} f(\gamma^{\m}(\tau;(t,x,v))) \psi(\m) \dd \m
\end{equation*}
with
\begin{equation*}
	K_\tau(s,y,w) := 2^{d}\tau^{-\homd} \psi\left( \frac{s}{\tau^2}, \A_{\frac{s}{\tau^2}}(\tau)^{-1}\binom{y}{w} \right).
\end{equation*}
The operator $T_{K_\tau}$ is called kinetic mollification.

In order to use the kinetic mollification to study regularity properties of \(f\), we need to understand the mollification defect \(f-T_{K_\tau}f\). Let \(f\in \cS(\R^{1+2d})\) and assume
\begin{equation*}
	(\partial_t + v \cdot \nabla_x)f = \nabla_v \cdot S_0 + S_1
\end{equation*}
with \(S_0 \in \cS(\R^{1+2d};\R^d)\) and \(S_1 \in \cS(\R^{1+2d})\).

For every \((t,x,v)\in\R^{1+2d}\), the fundamental theorem of calculus gives
\begin{align*}
	f(t,x,v)-[T_{K_\tau}(f)](t,x,v)
	 & = \int_{\R^{1+2d}} \bigl(f(t,x,v)-f(\gamma^{\m}(\tau;(t,x,v)))\bigr)\psi(\m)\,\dd \m                    \\
	 & = - \int_{\R^{1+2d}} \int_0^\tau \frac{\dd}{\dd r} f(\gamma^{\m}(r;(t,x,v))) \dd r \ \psi(\m)  \dd \m .
\end{align*}
Property \textbf{(M1)} yields
\begin{align*}
	\frac{\dd}{\dd r} f(\gamma^{\m}(r;(t,x,v)))
	 & =
	2m_0r\,[(\partial_t + v\cdot \nabla_x) f](\gamma^{\m}(r;(t,x,v))) \\
	 & \quad
	+\dot{\gamma}^{\m}_v(r;(t,x,v))\cdot [\nabla_v f](\gamma^{\m}(r;(t,x,v))).
\end{align*}
Hence
\begin{equation*}
	f(t,x,v)-[T_{K_\tau}(f)](t,x,v)= I_0(t,x,v)+I_1(t,x,v)+I_{\rm forc}(t,x,v),
\end{equation*}
where
\begin{align*}
	I_0(t,x,v)
	 & :=
	-\int_0^\tau \int_{\R^{1+2d}}
	2m_0r\,[\nabla_v \cdot S_0](\gamma^{\m}(r;(t,x,v))) \psi(\m)\,\dd \m\,\dd r, \\
	I_1(t,x,v)
	 & :=
	-\int_0^\tau \int_{\R^{1+2d}}
	2m_0r\,S_1(\gamma^{\m}(r;(t,x,v))) \psi(\m)\,\dd \m\,\dd r,                  \\
	I_{\rm forc}(t,x,v)
	 & :=
	-\int_0^\tau \int_{\R^{1+2d}}
	\dot{\gamma}^{\m}_v(r;(t,x,v))\cdot [\nabla_v f](\gamma^{\m}(r;(t,x,v))) \psi(\m)\,\dd \m\,\dd r .
\end{align*}

Fix \(r>0\) and perform the change of variables
\[
	\tilde{\m}=\gamma^{\m}(r;(t,x,v)).
\]
If we write
\[
	(s,y,w):=(t,x,v)^{-1}\circ \tilde{\m},
\]
then
\begin{equation*}
	s=\tilde m_0-t=m_0r^2,
	\qquad
	\binom{y}{w}
	=
	\binom{\tilde m_1-x-sv}{\tilde m_2-v}
	=
	\A_{m_0}(r)\binom{m_1}{m_2}.
\end{equation*}
Therefore
\begin{equation*}
	m_0=\frac{s}{r^2},
	\qquad
	\binom{m_1}{m_2}
	=
	\A_{\frac{s}{r^2}}(r)^{-1}\binom{y}{w}.
\end{equation*}
Using \textbf{(M2)}, i.e. \(|\det \A_{m_0}(r)|=2^{-d}r^{4d}\), we obtain
\[
	\dd \m = 2^d r^{-\homd}\,\dd \tilde{\m}.
\]

Applying this change of variables to \(I_0\), we obtain
\begin{align*}
	I_0(t,x,v)
	 & =
	-\int_0^\tau \int_{\R^{1+2d}}
	2^{d+1}s\,r^{-\homd-1}
		[\nabla_{\tilde m_2}\cdot S_0](\tilde{\m})
	\psi\left(
	\frac{s}{r^2},
	\A_{\frac{s}{r^2}}(r)^{-1}\binom{y}{w}
	\right)
	\dd \tilde{\m}\,\dd r.
\end{align*}
Integrating by parts in \(\tilde m_2\) gives
\begin{align*}
	I_0(t,x,v)
	 & =
	\int_0^\tau \int_{\R^{1+2d}}
	2^{d+1}s\,r^{-\homd-1}
	S_0(\tilde{\m})       \\
	 & \qquad \cdot
	\left[
		[\nabla_{(m_1,m_2)}\psi]^T
		\left(
		\frac{s}{r^2},
		\A_{\frac{s}{r^2}}(r)^{-1}\binom{y}{w}
		\right)
		\left(\A_{\frac{s}{r^2}}(r)^{-1}\right)_{\cdot;2}
		\right]
	\dd \tilde{\m}\,\dd r \\
	 & =
	\int_0^\tau [T_{\kpi_r}(S_0)](t,x,v)\,\dd r.
\end{align*}

For the lower-order source term \(S_1\), no integration by parts is needed. The same change of variables gives
\begin{align*}
	I_1(t,x,v)
	 & =
	-\int_0^\tau \int_{\R^{1+2d}}
	2^{d+1}s\,r^{-\homd-1}
	S_1(\tilde{\m})
	\psi\left(
	\frac{s}{r^2},
	\A_{\frac{s}{r^2}}(r)^{-1}\binom{y}{w}
	\right)
	\dd \tilde{\m}\,\dd r \\
	 & =
	\int_0^\tau [T_{\widetilde K_r}(S_1)](t,x,v)\,\dd r.
\end{align*}

Finally, since
\[
	\dot{\gamma}^{\m}_v(r;(t,x,v))=\cF_{m_0}(r)\binom{m_1}{m_2},
\]
the same change of variables yields
\begin{align*}
	I_{\rm forc}(t,x,v)
	 & =
	-\int_0^\tau \int_{\R^{1+2d}}
	2^d r^{-\homd}\,
	[\nabla_v f](\tilde{\m}) \\
	 & \qquad \cdot
	\psi\left(
	\frac{s}{r^2},
	\A_{\frac{s}{r^2}}(r)^{-1}\binom{y}{w}
	\right)
	\left[
		\cF_{\frac{s}{r^2}}(r)\A_{\frac{s}{r^2}}(r)^{-1}\binom{y}{w}
		\right]
	\dd \tilde{\m}\,\dd r    \\
	 & =
	\int_0^\tau [T_{\kforc_r}(\nabla_v f)](t,x,v)\,\dd r.
\end{align*}

We have thus shown the representation formulas
\begin{equation} \label{eq:TrfK}
	[T_{K_\tau}(f)](t,x,v) = \int_{\R^{1+2d}} f(\m) K_\tau\left( (t,x,v)^{-1} \circ \m \right) \dd \m
\end{equation}
and
\begin{equation} \label{eq:repfTtauf}
	[f-T_{K_\tau} f](t,x,v)
	=
	\int_0^\tau
	\Big(
	[T_{\kpi_r} (S_0)](t,x,v)
	+
	[T_{\widetilde{K}_r}(S_1)](t,x,v)
	+
	[T_{\kforc_r} (\nabla_v f)](t,x,v)
	\Big)
	\dd r.
\end{equation}
Here, the kinetic kernels are defined for \(r,\tau \in (0,\infty)\) as follows:
\medskip
\begin{enumerate}
	\item \(K_\tau \colon \R^{1+2d} \to \R\),
	      \begin{equation*}
		      K_\tau(s,y,w)
		      =
		      2^{d}\tau^{-\homd}
		      \psi\left(
		      \frac{s}{\tau^2},
		      \A_{\frac{s}{\tau^2}}(\tau)^{-1}\binom{y}{w}
		      \right),
	      \end{equation*}
	\item \(\widetilde{K}_r \colon \R^{1+2d} \to \R\),
	      \begin{equation*}
		      \widetilde{K}_r(s,y,w)
		      =
		      -2^{d+1}s\,r^{-\homd-1}
		      \psi\left(
		      \frac{s}{r^2},
		      \A_{\frac{s}{r^2}}(r)^{-1}\binom{y}{w}
		      \right),
	      \end{equation*}
	\item \(\kforc_r \colon \R^{1+2d} \to \R^d\),
	      \begin{equation*}
		      \kforc_r(s,y,w)
		      =
		      -2^d r^{-\homd}\,
		      \psi\!\left(
		      \frac{s}{r^2},
		      \A_{\frac{s}{r^2}}(r)^{-1}\binom{y}{w}
		      \right)
		      \Bigl[
			      \cF_{\frac{s}{r^2}}(r)\A_{\frac{s}{r^2}}(r)^{-1}\binom{y}{w}
			      \Bigr],
	      \end{equation*}
	\item and \(\kpi_r \colon \R^{1+2d} \to \R^d\),
	      \begin{equation*}
		      \kpi_r(s,y,w)
		      =
		      2^{d+1}s\,r^{-\homd-1}
			      [\nabla_{(m_1,m_2)}\psi]^T\left(
		      \frac{s}{r^2},
		      \A_{\frac{s}{r^2}}(r)^{-1}\binom{y}{w}
		      \right)
		      \left(\A_{\frac{s}{r^2}}(r)^{-1}\right)_{\cdot;2}.
	      \end{equation*}
\end{enumerate}

\subsection{\texorpdfstring{$\L^p$}{Lp} bounds for the kinetic mollification}

Even though the kinetic convolution is not commutative we still have Young-type inequalities due to the isometry of the kinetic shift; compare \cite[Section 2.4]{MR4049224}.

\begin{lemma} \label{lem:young}
	Let $1\le p,q,\theta \le \infty$ with $\frac{1}{q}+1 = \frac{1}{\theta}+\frac{1}{p}$. For any $J \in \L^{\theta}(\R^{1+2d};\R^N)$ we have
	\begin{equation*}
		\norm{T_J(f)}_{\L^q} \le \norm{J}_{\L^{\theta}} \norm{f}_{\L^p}.
	\end{equation*}
	for all $f \in \L^p(\R^{1+2d};\R^N)$.
\end{lemma}

We start with localisation and estimation of the kernels $K,\widetilde K,\kpi$ and $\kforc$. We define
\[
	B_r^{\rm kin}=\bigl\{(s,y,w)\in\R^{1+2d}: |s|<r^2,\ |y|<r^3,\ |w|<r\bigr\}.
\]

\begin{lemma} \label{lem:KforcKpiloc}
	There exists $C>0$ such that for all $r>0$,
	\begin{equation}\label{eq:supp-size-kernels}
		\begin{aligned}
			\supp K_r \cup \supp \widetilde K_r \cup \supp \kpi_r \cup \supp \kforc_r & \subset B_{Cr}^{\rm kin}, \\
			|K_r|+|\kpi_r|+|\kforc_r|                                                 & \lesssim r^{-\homd},      \\
			|\widetilde K_r|                                                          & \lesssim r^{-\homd+1},    \\
			|\nabla_y K_r|+|\nabla_y \kpi_r|+|\nabla_y \kforc_r|                      & \lesssim r^{-\homd-3},    \\
			|\nabla_y \widetilde K_r|                                                 & \lesssim r^{-\homd-2}.
		\end{aligned}
	\end{equation}
	with $B_r^{\rm kin}:=\bigl\{(s,y,w)\in\R^{1+2d}: |s|<r^2,\ |y|<r^3,\ |w|<r\bigr\}$. Here $\nabla_y$ denotes differentiation in the spatial kernel variable $y$ in $K(s,y,w)$.
\end{lemma}

\begin{proof}
	If one of the kernels is nonzero, then
	\[
		a:=\frac{s}{r^2}\in(-2,-1),\qquad
		b:=\A_a(r)^{-1}\binom{y}{w}\in B_1(0)\times B_1(0),
	\]
	because $\psi$ is supported in $(-2,-1)\times B_1(0)\times B_1(0)$.
	We have
	\[
		\A_a(r)=\D_a \cW(r)\D_a^{-1}
		=
		\begin{pmatrix}
			g_1(r)                       & a\,g_2(r)              \\[1mm]
			a^{-1}\frac{\dot g_1(r)}{2r} & \frac{\dot g_2(r)}{2r}
		\end{pmatrix},
	\]
	with $a\in[-2,-1]$ and the explicit formulas for $g_1,g_2$ give
	\[
		|g_1(r)|+|g_2(r)|\lesssim r^3,
		\qquad
		\left|\frac{\dot g_1(r)}{2r}\right|
		+
		\left|\frac{\dot g_2(r)}{2r}\right|
		\lesssim r.
	\]
	Hence $|(y,w)|=|\A_a(r)b|$ implies $|y|\lesssim r^3$, $|w|\lesssim r$, while clearly $|s|\lesssim r^2$, so the support inclusion follows.

	For $K_r$ we use its definition:
	\[
		K_r(s,y,w)
		=
		2^d r^{-\homd}\,
		\psi\!\left(\frac{s}{r^2}, \A_{\frac{s}{r^2}}(r)^{-1}\binom{y}{w}\right).
	\]
	Thus $|K_r|\lesssim r^{-\homd}$. Differentiating with respect to $y$ only hits the argument of $\psi$, so one gains one factor
	\[
		\left|\left(\A_{\frac{s}{r^2}}(r)^{-1}\right)_{\cdot;1}\right|\lesssim r^{-3}
	\]
	again by \textbf{(M3)}, and hence $|\nabla_y K_r|\lesssim r^{-\homd-3}$.

	For $\widetilde K_r$ we use
	\[
		\widetilde K_r(s,y,w)
		=
		-2^{d+1}s\,r^{-\homd-1}
		\psi\!\left(\frac{s}{r^2}, \A_{\frac{s}{r^2}}(r)^{-1}\binom{y}{w}\right).
	\]
	On the support we have $|s|\lesssim r^2$, hence $|\widetilde K_r|\lesssim r^{-\homd+1}$.
	Differentiating with respect to $y$ only hits the argument of $\psi$, so we gain one factor
	\[
		\left|\left(\A_{\frac{s}{r^2}}(r)^{-1}\right)_{\cdot;1}\right|\lesssim r^{-3},
	\]
	and therefore $|\nabla_y \widetilde K_r|\lesssim r^{-\homd-2}$.

	For $\kpi_r$ we use its definition:
	\[
		\kpi_r(s,y,w)
		=
		2^{d+1} s\, r^{-\homd-1}
			[\nabla_{(m_1,m_2)}\psi]^T\!\left( \frac{s}{r^2}, \A_{\frac{s}{r^2}}(r)^{-1}\binom{y}{w} \right)
		\left(\A_{\frac{s}{r^2}}(r)^{-1}\right)_{\cdot;2}.
	\]
	On the support we have $|s|\lesssim r^2$, and by \textbf{(M3)}
	\[
		\left|\left(\A_{\frac{s}{r^2}}(r)^{-1}\right)_{\cdot;2}\right|\lesssim r^{-1}.
	\]
	Therefore $|\kpi_r|\lesssim r^{-\homd}$. Differentiating with respect to $y$ only hits the argument of $\psi$, so we gain one factor
	\[
		\left|\left(\A_{\frac{s}{r^2}}(r)^{-1}\right)_{\cdot;1}\right|\lesssim r^{-3}
	\]
	again by \textbf{(M3)}, and hence $|\nabla_y\kpi_r|\lesssim r^{-\homd-3}$.

	For $\kforc_r$ we recall
	\[
		\kforc_r(s,y,w)
		=
		-2^d r^{-\homd}\,
		\psi\!\left(\frac{s}{r^2}, \A_{\frac{s}{r^2}}(r)^{-1}\binom{y}{w}\right)
		\Bigl[\cF_{\frac{s}{r^2}}(r)\A_{\frac{s}{r^2}}(r)^{-1}\binom{y}{w}\Bigr].
	\]
	Since $s/r^2\in[-2,-1]$ on the support, the explicit formula for $\cF_{m_0}(r)$ shows
	\[
		\sup_{r>0}\sup_{m_0\in[-2,-1]} |\cF_{m_0}(r)|<\infty.
	\]
	Moreover, on the support we have
	\[
		\left|\A_{\frac{s}{r^2}}(r)^{-1}\binom{y}{w}\right|\le 1.
	\]
	Thus $|\kforc_r|\lesssim r^{-\homd}$. If we differentiate with respect to $y$, either the derivative falls on $\psi$, or on the linear factor
	\(
	\cF_{\frac{s}{r^2}}(r)\A_{\frac{s}{r^2}}(r)^{-1}\binom{y}{w}.
	\)
	In both cases we gain one factor
	\[
		\left|\left(\A_{\frac{s}{r^2}}(r)^{-1}\right)_{\cdot;1}\right|\lesssim r^{-3},
	\]
	so $|\nabla_y\kforc_r|\lesssim r^{-\homd-3}$.
\end{proof}

\begin{coro}\label{cor:dxhK}
	For all $r>0$ and $h\in\R^d$,
	\begin{equation*}
		\|\Delta_x^h K_r\|_{\L^1}
		+\|\Delta_x^h \kpi_r\|_{\L^1}
		+\|\Delta_x^h \kforc_r\|_{\L^1}
		\lesssim \min\{1,\,|h|\,r^{-3}\}.
	\end{equation*}
	Here $\Delta_x^h$ acts on the spatial kernel variable $y$ as $[\Delta_x^h J](s,y,w) = J(s,y+h,w)-J(s,y,w)$.
\end{coro}

\begin{proof}
	For $J_r\in\{K_r,\kpi_r,\kforc_r\}$, Lemma \ref{lem:KforcKpiloc} gives
	\[
		\supp J_r\subset B_{Cr}^{\rm kin},\qquad
		|J_r|\lesssim r^{-\homd},\qquad
		|\nabla_y J_r|\lesssim r^{-\homd-3}.
	\]
	Since $|B_{Cr}^{\rm kin}|\lesssim r^{\homd}$, we obtain
	\[
		\|J_r\|_{\L^1}\lesssim 1,
		\qquad
		\|\nabla_y J_r\|_{\L^1}\lesssim r^{-3}.
	\]
	Hence
	\[
		\|\Delta_x^h J_r\|_{\L^1}
		\le
		\min\bigl\{2\|J_r\|_{\L^1},\,|h|\,\|\nabla_y J_r\|_{\L^1}\bigr\}
		\lesssim
		\min\{1,\,|h|\,r^{-3}\}.
	\]
\end{proof}

\begin{coro}\label{cor:dxtildeK}
	Let $1\le \theta\le\infty$ and let $\theta'$ be the H\"older conjugate exponent. Then for all $r>0$ and $h\in\R^d$,
	\begin{equation*}
		\|\widetilde K_r\|_{\L^\theta}
		\lesssim r^{1-\homd/\theta'},
		\qquad
		\|\nabla_y \widetilde K_r\|_{\L^\theta}
		\lesssim r^{-2-\homd/\theta'},
	\end{equation*}
	and
	\begin{equation*}
		\|\Delta_x^h \widetilde K_r\|_{\L^\theta}
		\lesssim
		r^{1-\homd/\theta'}\min\{1,\,|h|\,r^{-3}\}.
	\end{equation*}
\end{coro}

\begin{proof}
	By Lemma \ref{lem:KforcKpiloc},
	\[
		\supp \widetilde K_r\subset B_{Cr}^{\rm kin},
		\qquad
		|\widetilde K_r|\lesssim r^{-\homd+1},
		\qquad
		|\nabla_y \widetilde K_r|\lesssim r^{-\homd-2}.
	\]
	Since $|B_{Cr}^{\rm kin}|\lesssim r^{\homd}$, we obtain
	\[
		\|\widetilde K_r\|_{\L^\theta}
		\lesssim
		r^{-\homd+1}\,|B_{Cr}^{\rm kin}|^{1/\theta}
		\lesssim
		r^{1-\homd+\homd/\theta}
		=
		r^{1-\homd/\theta'}
	\]
	and similarly
	\[
		\|\nabla_y \widetilde K_r\|_{\L^\theta}
		\lesssim
		r^{-\homd-2}\,|B_{Cr}^{\rm kin}|^{1/\theta}
		=
		r^{-2-\homd/\theta'}.
	\]
	The difference estimate follows from
	\[
		\|\Delta_x^h \widetilde K_r\|_{\L^\theta}
		\le
		\min\bigl\{2\|\widetilde K_r\|_{\L^\theta},\,|h|\,\|\nabla_y \widetilde K_r\|_{\L^\theta}\bigr\}.
	\]
\end{proof}

\begin{lemma}\label{lem:commuteDx}
	Let $J\in\mathcal S(\R^{1+2d})$ and $f\in\mathcal S(\R^{1+2d})$.
	Then
	\[
		D_x^{\frac13}\,T_J f \;=\; T_{D_y^{\frac13}J}\,f.
	\]
\end{lemma}

\begin{proof}
	Since $f,J\in \mathcal S(\R^{1+2d})$, also $T_Jf\in \mathcal S(\R^{1+2d})$.
	We use the standard singular-integral representation of $D^{1/3}$: for every
	$\phi\in\mathcal S(\R^d)$,
	\[
		D^{\frac13}\phi(z)
		=
		c_d\int_{\R^d}\frac{2\phi(z)-\phi(z+h)-\phi(z-h)}{|h|^{d+\frac13}}\,\dd h,
	\]
	where $c_d>0$ is the appropriate normalisation constant; see \cite{MR3613319}. To justify Fubini
	below, we may first truncate to $\epsilon<|h|<R$ and then pass to the limit.

	For $h\in\R^d$, define the symmetric second differences
	\[
		\delta_h^x G(t,x,v):=2G(t,x,v)-G(t,x+h,v)-G(t,x-h,v),
	\]
	and
	\[
		\delta_h^y J(s,y,w):=2J(s,y,w)-J(s,y+h,w)-J(s,y-h,w).
	\]
	We claim that
	\[
		\delta_h^x(T_Jf)=T_{\delta_h^y J}f.
	\]
	Indeed, fixing $(t,x,v)$ and using the definition of $T_J$,
	\begin{align*}
		\delta_h^x(T_Jf)(t,x,v)
		 & =2T_Jf(t,x,v)-T_Jf(t,x+h,v)-T_Jf(t,x-h,v).
	\end{align*}
	In the term $T_Jf(t,x+h,v)$, make the change of variables $y\mapsto y+h$ and in
	the term $T_Jf(t,x-h,v)$, make the change of variables $y\mapsto y-h$. This gives
	\begin{align*}
		\delta_h^x(T_Jf)(t,x,v)
		 & =\int_{\R^{1+2d}} f(t+s,x+y+sv,v+w)                           \\
		 & \qquad\qquad\cdot
		\bigl(2J(s,y,w)-J(s,y-h,w)-J(s,y+h,w)\bigr)\,\dd s\,\dd y\,\dd w \\
		 & = T_{\delta_h^y J}f(t,x,v).
	\end{align*}

	Now apply the singular-integral formula in the $x$ variable:
	\begin{align*}
		D_x^{\frac13}[T_Jf](t,x,v)
		 & =
		c_d\int_{\R^d}\frac{\delta_h^x[T_Jf](t,x,v)}{|h|^{d+\frac13}}\,\dd h
		=
		c_d\int_{\R^d}\frac{[T_{\delta_h^y J}f](t,x,v)}{|h|^{d+\frac13}}\,\dd h \\
		 & =
		[T_{\displaystyle \left(\,c_d\int_{\R^d}\frac{\delta_h^y J}{|h|^{d+\frac13}}\,\dd h \right)}](t,x,v) =[T_{D_y^{\frac13}J}f](t,x,v)
	\end{align*}
	by bilinearity.
\end{proof}

\section{Transfer-of-regularity in H\"ormander's Besov spaces} \label{sec:torHB}

We first prove the transfer-of-regularity estimate in a Besov-type norm in the
$x$-variable. These spaces were also used in H\"ormander's original work \cite{hormander_hypoelliptic_1967} in the case $p =2$. The additional term $S_1$ is measured in the critical Lebesgue
space corresponding to one derivative in the kinetic homogeneous dimension
$\homd=4d+2$.

\begin{theorem}\label{thm:kintorBesov}
	Let $p \in [\frac{\homd}{\homd-1},\infty]$ and let $q\in[1,\infty]$ be defined by
	\[
		\frac1q = \frac1p + \frac1{\homd}.
	\]
	Then there exists $C=C(d,p)>0$ such that for every
	\[
		f\in \cS(\R^{1+2d}) \cap \L^p(\R^{1+2d}),
		\qquad
		\nabla_v f\in \L^p(\R^{1+2d};\R^d),
	\]
	satisfying
	\[
		(\partial_t+v\cdot\nabla_x)f=\nabla_v\!\cdot S_0 + S_1
	\]
	with
	\[
		S_0\in \cS(\R^{1+2d};\R^d)\cap\L^p(\R^{1+2d};\R^d),
		\qquad
		S_1\in \cS(\R^{1+2d})\cap\L^q(\R^{1+2d}),
	\]
	we have
	\begin{equation}\label{eq:kintorBesov}
		\sup_{h \neq 0} \frac{\norm{\Delta_x^h f}_{\L^p_{t,x,v}}}{\abs{h}^\frac{1}{3}}
		\le
		C\bigl(
		\|S_0\|_{\L^p_{t,x,v}}
		+\|\nabla_v f\|_{\L^p_{t,x,v}}
		+\|S_1\|_{\L^q_{t,x,v}}
		\bigr).
	\end{equation}
\end{theorem}

\begin{remark}
	In the case $S_1=0$, estimates of this type go back to H\"ormander
	\cite{hormander_hypoelliptic_1967}, who treated the Hilbert-space case
	$p=2$ and obtained a non-sharp gain $s<\frac13$. As observed in
	\cite{albritton2021variational}, H\"ormander's method can in fact be
	pushed to the sharp exponent $s=\frac13$. We are not aware of an adaptation
	of this argument to the case $p\neq 2$.
\end{remark}

\begin{remark}
	The exponent
	\[
		q=\frac{\homd\,p}{\homd+p}=\frac{(4d+2)p}{4d+2+p}
	\]
	is dictated by the scaling of $\widetilde K_r$. When $S_1=0$, the same proof works
	for all $p\in[1,\infty]$, since only the $\L^1$ kernel bounds from
	Corollary \ref{cor:dxhK} are used.
\end{remark}

\begin{proof}[Proof of Theorem \ref{thm:kintorBesov}]
	Set $\theta := \frac{\homd}{\homd-1}$. Then $\theta'=\homd$ and
	\[
		\frac1p + 1 = \frac1\theta + \frac1q.
	\]
	For $\tau>0$ we write
	\[
		f = f-T_{K_\tau}(f) + T_{K_\tau}(f).
	\]
	Using \eqref{eq:repfTtauf}, Minkowski's inequality, Lemma \ref{lem:young},
	Corollary \ref{cor:dxhK}, and Corollary \ref{cor:dxtildeK}, we obtain for every
	$h\in\R^d$
	\begin{align*}
		 & \norm{\Delta_x^h (f-T_{K_\tau}(f))}_{\L^p} \\
		 & \le
		\int_0^\tau
		\Bigl(
		\norm{\Delta_x^h \kpi_r}_{\L^1}\,\|S_0\|_{\L^p}
		+\norm{\Delta_x^h \kforc_r}_{\L^1}\,\|\nabla_v f\|_{\L^p}
		+\norm{\Delta_x^h \widetilde K_r}_{\L^\theta}\,\|S_1\|_{\L^q}
		\Bigr)\dd r                                   \\
		 & \lesssim
		\int_0^\tau \min\{1,\,|h|\,r^{-3}\}\dd r \,
		\bigl(
		\|S_0\|_{\L^p}
		+\|\nabla_v f\|_{\L^p}
		+\|S_1\|_{\L^q}
		\bigr).
	\end{align*}
	Since
	\[
		\int_0^\infty \min\{1,\,|h|\,r^{-3}\}\dd r \lesssim |h|^{\frac13},
	\]
	this yields
	\[
		\norm{\Delta_x^h (f-T_{K_\tau}(f))}_{\L^p}
		\lesssim
		|h|^{\frac13}
		\bigl(
		\|S_0\|_{\L^p}
		+\|\nabla_v f\|_{\L^p}
		+\|S_1\|_{\L^q}
		\bigr).
	\]
	Moreover, by \eqref{eq:defTJ},
	\[
		\Delta_x^h T_{K_\tau}(f)=T_{\Delta_x^{-h}K_\tau}(f),
	\]
	where $\Delta_x^{-h}$ acts on the spatial kernel variable. Therefore, by
	Lemma \ref{lem:young} and Corollary \ref{cor:dxhK},
	\[
		\norm{\Delta_x^h T_{K_\tau}(f)}_{\L^p}
		\le
		\norm{\Delta_x^{-h}K_\tau}_{\L^1}\norm{f}_{\L^p}
		=
		\norm{\Delta_x^hK_\tau}_{\L^1}\norm{f}_{\L^p}
		\lesssim
		\min\{1,\abs{h}\tau^{-3}\}\norm{f}_{\L^p}.
	\]
	Hence
	\[
		\sup_{h\neq 0}\frac{\norm{\Delta_x^h T_{K_\tau}(f)}_{\L^p}}{\abs{h}^{\frac13}}
		\lesssim
		\sup_{h\neq 0}\frac{\min\{1,\abs{h}\tau^{-3}\}}{\abs{h}^{\frac13}}\norm{f}_{\L^p}
		\lesssim
		\tau^{-1}\norm{f}_{\L^p},
	\]
	since $\min\{1,a\}\le a^{1/3}$ for all $a\ge0$.
	Letting $\tau\to\infty$ proves \eqref{eq:kintorBesov}.
\end{proof}

\section{Transfer-of-regularity in Sobolev spaces} \label{sec:torLp}

In the pure divergence case, the Besov estimate above already implies
$f \in \L^p_{t,v}\Wdot^{s,p}_x$ for every $s \in (0,\frac{1}{3})$.
Indeed,
\begin{equation*}
	\|f\|_{\L^p_{t,v}\Wdot^{s,p}_x}^p \simeq \int_{\R^d} \frac{\norm{\Delta_x^h f }_{\L^p_{t,x,v}}^p}{|h|^{d+sp}} \dd h,
\end{equation*}
see \cite{MR2944369}. We need, however, to invest also $f \in \L^p_{t,x,v}$ on the right-hand side of the estimate. To reach the sharp regularity $s = \frac{1}{3}$ and a
scale-invariant estimate, we use Littlewood--Paley theory together with maximal
operators on the kinetic group.

\begin{theorem}\label{thm:kintorLpGeneral}
	Let $p \in (\frac{\homd}{\homd-1},\infty)$ and let $q\in(1,\infty)$ be defined by
	\[
		\frac1q = \frac1p + \frac1{\homd}.
	\]
	Then there exists $C=C(d,p)>0$ such that for every
	\[
		f\in \cS(\R^{1+2d}) \cap \L^p(\R^{1+2d}),
		\qquad
		\nabla_v f\in \L^p(\R^{1+2d};\R^d),
	\]
	satisfying
	\[
		(\partial_t+v\cdot\nabla_x)f=\nabla_v\!\cdot S_0 + S_1
	\]
	with
	\[
		S_0\in \mathcal S(\R^{1+2d};\R^d)\cap\L^p(\R^{1+2d};\R^d),
		\qquad
		S_1\in \mathcal S(\R^{1+2d})\cap\L^q(\R^{1+2d}),
	\]
	we have
	\begin{equation}\label{eq:kintorLp-general}
		\|D_x^{\frac13} f\|_{\L^p_{t,x,v}}
		\le C\bigl(\|S_0\|_{\L^p_{t,x,v}}+\|\nabla_v f\|_{\L^p_{t,x,v}}+\|S_1\|_{\L^q_{t,x,v}}\bigr).
	\end{equation}
\end{theorem}

\begin{remark}
	The estimate in Theorem \ref{thm:kintorLpGeneral} is not claimed here as new. For $S_1=0$, the sharp $\L^p$ transfer estimate goes back to Bouchut \cite{bouchut_hypoelliptic_2002}; related estimates with lower-order source terms and their formulation in kinetic Sobolev spaces are contained in \cite{AIN,an_kfpLp_2025}. The novelty of the present argument is the route to the estimate. Instead of using the Fourier symbol of $\partial_t+v\cdot\nabla_x$, characteristics after Fourier transform, or the fundamental solution of the Kolmogorov equation, we prove the estimate directly from the physical-space identity \eqref{eq:repfTtauf}. In this sense Theorem \ref{thm:kintorLpGeneral} is used as a test case for the trajectory method: the proof shows that the sharp Bouchut--H\"ormander gain follows from critical kinetic trajectories, kernel bounds, and vector-valued maximal estimates on the kinetic group.
\end{remark}

\begin{proof}[Proof of Theorem \ref{thm:kintorLpGeneral}]
	\noindent
	\textbf{Step 1: Littlewood--Paley decomposition in $x$.}
	Choose $\eta,\widetilde \eta\in \C_c^\infty(\R^d\setminus\{0\})$, both symmetric, such that
	\[
		1=\sum_{j\in\Z}\eta_j(\varphi),\qquad \varphi\neq0,
	\]
	with $\eta_j(\varphi):=\eta(2^{-j}\varphi)$ and set $\widetilde\eta_j(\varphi):=\widetilde\eta(2^{-j}\varphi)$, where $\widetilde\eta_j = 1\quad\text{on }\supp\eta_j$. Furthermore, we set
	\[
		P_j:=\eta_j(D_x),\qquad \widetilde P_j:=\widetilde\eta_j(D_x).
	\]
	For Schwartz functions we have the homogeneous Littlewood--Paley equivalence
	\begin{equation}\label{eq:LP-hom}
		\|D_x^{\frac13}u\|_{\L^p}
		\simeq
		\left\|
		\left(\sum_{j\in\Z}2^{\frac{2j}{3}}|P_j u|^2\right)^{1/2}
		\right\|_{\L^p},
	\end{equation}
	by \cite[Theorem 6.2.7]{MR3243741} and also, for every $a\in(1,\infty)$,
	\begin{equation}\label{eq:LP-square}
		\left\|
		\left(\sum_{j\in\Z}|\widetilde P_j g|^2\right)^{1/2}
		\right\|_{\L^a}
		\lesssim_a \|g\|_{\L^a}
	\end{equation}
	by \cite[Theorem 6.1.2]{MR3243734}.

	Moreover, there exists a family
	$\Psi_j=(\Psi_{j,1},\dots,\Psi_{j,d})\in \mathcal S(\R^d;\R^d)$ such that
	\[
		\sup_{j\in\Z}\|\Psi_j\|_{\L^1(\R^d)}<\infty
	\]
	and
	\begin{equation}\label{eq:Pj-grad}
		P_j g
		=
		2^{-j} \sum_{\ell = 1}^d \partial_{x_\ell}\bigl(\Psi_{j,\ell} *_x \widetilde P_j g\bigr)
		\qquad\text{for all }g\in \mathcal S'(\R^{1+2d}).
	\end{equation}
	In fact, $\widehat{\Psi_{j,\ell}}(\varphi) := -\,i\,2^j\,\frac{\varphi_\ell}{|\varphi|^2}\,\eta_j(\varphi)$, $\varphi\in\R^d \setminus \{ 0 \}$ for $\ell = 1,\dots, d$ and $j \in \Z$ do the job.

	Since $\check \eta_j$ and $\Psi_j$ are dyadic rescalings of Schwartz kernels, the standard maximal-function domination \cite[Corollary 2.1.12]{MR3243734} gives
	\begin{equation}\label{eq:max-x}
		\sup_{j \in \Z}\left(|\check\eta_j *_x g|+|\Psi_j *_x g|\right)
		\lesssim M_x(|g|),
	\end{equation}
	where $M_x$ denotes the Euclidean Hardy--Littlewood maximal operator in the $x$-variable only:
	\[
		[M_x (g)](t,x,v):=\sup_{\rho>0} \frac{1}{|B_\rho(x)|} \int_{B_\rho(x)} |g(t,y,v)|\dd y.
	\]
	We shall also use the vector-valued Fefferman--Stein inequality \cite{MR284802}, see also \cite[Theorem 5.6.6]{MR3243734}, in the $x$-variable: for every $a\in(1,\infty)$,
	\begin{equation}\label{eq:FS-x}
		\left\|
		\left(\sum_j |M_x g_j|^2\right)^{1/2}
		\right\|_{\L^a}
		\lesssim_a
		\left\|
		\left(\sum_j |g_j|^2\right)^{1/2}
		\right\|_{\L^a}.
	\end{equation}

	\smallskip

	\noindent
	\textbf{Step 2: kinetic maximal operators.}
	Recall
	\[
		B_r^{\rm kin}=\bigl\{(s,y,w)\in\R^{1+2d}: |s|<r^2,\ |y|<r^3,\ |w|<r\bigr\}
	\]
	and $|B_r^{\rm kin}|\simeq r^{\homd}$ with $\homd=4d+2$. We define
	\[
		[M_{\rm kin} (g)](t,x,v):=
		\sup_{r>0}\frac1{|B_r^{\rm kin}|}\int_{B_r^{\rm kin}} |g((t,x,v)\circ \m)|\dd \m
	\]
	and
	\[
		[M_{{\rm kin},1} (g)](t,x,v):=
		\sup_{r>0} r\,\frac1{|B_r^{\rm kin}|}\int_{B_r^{\rm kin}} |g((t,x,v)\circ \m)|\dd \m.
	\]

	Set
	\[
		\rho_{\rm box}(s,y,w):=\max\{|s|^{1/2},|y|^{1/3},|w|\}.
	\]
	Since
	\[
		(s,y,w)^{-1}=(-s,-y+sw,-w),
	\]
	we have $\rho_{\rm box}(\m^{-1})\lesssim \rho_{\rm box}(\m)$, and from
	\[
		(s,y,w)\circ(s',y',w')=(s+s',\,y+y'+s'w,\,w+w')
	\]
	we check that
	\[
		\rho_{\rm box}(\xi\circ\eta)\lesssim \rho_{\rm box}(\xi)+\rho_{\rm box}(\eta).
	\]
	Hence
	\[
		\rho_{\rm kin}(\m):=\rho_{\rm box}(\m)+\rho_{\rm box}(\m^{-1})
	\]
	is a homogeneous quasi-norm and
	\[
		\rho_{\rm box}(\m)\le \rho_{\rm kin}(\m)\lesssim \rho_{\rm box}(\m).
	\]
	In particular, the $\rho_{\rm kin}$-balls are comparable with $B_r^{\rm kin}$.
	By \cite[Propositions~1.2.3 and~1.2.4(2)]{MR3966452}, there exists a homogeneous
	quasi-norm $|\cdot|_{\mathrm h}$ satisfying the triangle inequality and equivalent to
	$\rho_{\rm kin}$. Therefore, for some $C\ge1$,
	\[
		B_{r/C}^{\mathrm h}\subset B_r^{\rm kin}\subset B_{Cr}^{\mathrm h},
		\qquad
		B_r^{\mathrm h}:=\{\m\in\R^{1+2d}:|\m|_{\mathrm h}<r\},
	\]
	and $\abs{B_r^{\mathrm h}}\simeq \abs{B_r^{\rm kin}}\simeq r^{\homd}$.

	Let $M_{\mathrm h}$ denote the non-centred Hardy--Littlewood maximal operator
	associated with the left-invariant metric
	\[
		d_{\mathrm h}(z,\zeta):=\abs{z^{-1}\circ \zeta}_{\mathrm h}.
	\]
	Since each set $z\circ B_r^{\rm kin}$ is contained in a $d_{\mathrm h}$-ball of radius
	$Cr$ and comparable measure, we have
	\[
		M_{\rm kin}g(z)\lesssim M_{\mathrm h}g(z).
	\]
	The vector-valued Fefferman--Stein inequality on spaces of homogeneous type now yields
	\begin{equation}\label{eq:FS-kin}
		\left\|
		\left(\sum_j |M_{\rm kin} h_j|^2\right)^{1/2}
		\right\|_{\L^p}
		\lesssim
		\left\|
		\left(\sum_j |h_j|^2\right)^{1/2}
		\right\|_{\L^p},
	\end{equation}
	see \cite[Theorem~1.2]{MR2542655}.

	For the order-one operator, we argue through a fractional integral instead of
	invoking a separate vector-valued fractional maximal theorem. Define
	\[
		[I_1 g](z):=
		\int_{\R^{1+2d}}\frac{|g(z\circ \m)|}{|\m|_{\mathrm h}^{\homd-1}}\,\dd \m
		=
		\int_{\R^{1+2d}}\frac{|g(\zeta)|}{|z^{-1}\circ \zeta|_{\mathrm h}^{\homd-1}}\,\dd \zeta.
	\]
	Because $B_r^{\rm kin}\subset B_{Cr}^{\mathrm h}$, we have $|\m|_{\mathrm h}\lesssim r$
	on $B_r^{\rm kin}$. Using $\abs{B_r^{\rm kin}}\simeq r^{\homd}$, we obtain
	\[
		r\,\frac1{|B_r^{\rm kin}|}\int_{B_r^{\rm kin}}|g(z\circ \m)|\,\dd \m
		\lesssim
		\int_{B_r^{\rm kin}}\frac{|g(z\circ \m)|}{|\m|_{\mathrm h}^{\homd-1}}\,\dd \m
		\le I_1 g(z),
	\]
	hence
	\[
		M_{{\rm kin},1}g\lesssim I_1 g.
	\]

	Now set
	\[
		H(z):=\left(\sum_j |h_j(z)|^2\right)^{1/2}.
	\]
	By Minkowski's integral inequality,
	\[
		\left(\sum_j |I_1 h_j(z)|^2\right)^{1/2}
		\le
		\int_{\R^{1+2d}}\frac{\left(\sum_j |h_j(z\circ \m)|^2\right)^{1/2}}{|\m|_{\mathrm h}^{\homd-1}}\,\dd \m
		=
		I_1 H(z).
	\]
	Finally, if $\phi\in \L^{p'}$, then, using the definition of $I_1$ with absolute values,
	\begin{align*}
		\left|\int_{\R^{1+2d}} I_1 g(z)\,\phi(z)\,\dd z\right|
		 & \le
		\int_{\R^{1+2d}} I_1 g(z)\,|\phi(z)|\,\dd z \\
		 & =
		\iint_{\R^{1+2d}\times\R^{1+2d}}
		\frac{|g(\zeta)|\,|\phi(z)|}{|z^{-1}\circ \zeta|_{\mathrm h}^{\homd-1}}
		\,\dd \zeta\,\dd z                          \\
		 & \lesssim
		\|g\|_{\L^q}\|\phi\|_{\L^{p'}},
	\end{align*}
	by the Hardy--Littlewood--Sobolev inequality on homogeneous groups, i.e.\ \cite[Theorem~5.3.1]{MR3966452} with $\alpha=0$ and $\lambda=\homd-1$.
	Therefore
	\[
		\|I_1 g\|_{\L^p}\lesssim \|g\|_{\L^q},
		\qquad
		\frac1p=\frac1q-\frac1{\homd}.
	\]
	Combining this with the previous two displays, we conclude that
	\begin{equation}\label{eq:FS-kin-frac}
		\left\|
		\left(\sum_j |M_{{\rm kin},1} h_j|^2\right)^{1/2}
		\right\|_{\L^p}
		\lesssim
		\left\|
		\left(\sum_j |h_j|^2\right)^{1/2}
		\right\|_{\L^q}.
	\end{equation}

	\noindent
	\textbf{Step 3: pointwise domination for the kinetic kernels.}
	Fix $z=(t,x,v)\in\R^{1+2d}$. By \eqref{eq:Tjalternative},
	\begin{equation}\label{eq:TJ-left-translate}
		[T_{J_r}h](z)
		=
		\int_{\R^{1+2d}} J_r(s,y,w)\cdot h(z \circ (s,y,w)) \dd (s,y,w).
	\end{equation}
	For \(J_r\in\{\kpi_r,\kforc_r\}\), Lemma \ref{lem:KforcKpiloc} gives
	\[
		\supp J_r \subset B_{C_0r}^{\rm kin},
		\qquad
		|J_r|\lesssim r^{-\homd},
		\qquad
		|\partial_{y_\ell}J_r|\lesssim r^{-\homd-3}
	\]
	for some $C_0>0$ independent of $r$. Therefore
	\begin{align}
		|T_{J_r}h(z)|
		 & \le
		\|J_r\|_{\L^\infty}
		\int_{B_{C_0r}^{\rm kin}} |h(z\circ \m)| \dd \m
		\lesssim
		M_{\rm kin}(|h|)(z), \label{eq:max-bounds-J-1}
	\end{align}
	and similarly
	\begin{align}
		|T_{\partial_{y_\ell}J_r}h(z)|
		 & \le
		\|\partial_{y_\ell}J_r\|_{\L^\infty}
		\int_{B_{C_0r}^{\rm kin}} |h(z\circ \m)| \dd \m
		\lesssim
		r^{-3}M_{\rm kin}(|h|)(z). \label{eq:max-bounds-J-2}
	\end{align}

	For \(\widetilde K_r\), Lemma \ref{lem:KforcKpiloc} gives
	\[
		\supp \widetilde K_r \subset B_{C_0r}^{\rm kin},
		\qquad
		|\widetilde K_r|\lesssim r^{-\homd+1},
		\qquad
		|\partial_{y_\ell}\widetilde K_r|\lesssim r^{-\homd-2}.
	\]
	Hence
	\begin{align}
		|T_{\widetilde K_r}h(z)|
		 & \lesssim
		r^{-\homd+1}\int_{B_{C_0r}^{\rm kin}} |h(z\circ \m)| \dd \m
		\lesssim
		M_{{\rm kin},1}(|h|)(z), \label{eq:max-bounds-tilde-1}
	\end{align}
	and
	\begin{align}
		|T_{\partial_{y_\ell}\widetilde K_r}h(z)|
		 & \lesssim
		r^{-\homd-2}\int_{B_{C_0r}^{\rm kin}} |h(z\circ \m)| \dd \m
		\lesssim
		r^{-3}M_{{\rm kin},1}(|h|)(z). \label{eq:max-bounds-tilde-2}
	\end{align}

	\smallskip

	\noindent
	\textbf{Step 4: dyadic estimate for \(f-T_{K_\tau}f\).}
	By \eqref{eq:repfTtauf},
	\[
		f-T_{K_\tau}f
		=
		\int_0^\tau T_{\kpi_r}(S_0)\dd r
		+
		\int_0^\tau T_{\widetilde K_r}(S_1)\dd r
		+
		\int_0^\tau T_{\kforc_r}(\nabla_v f)\dd r .
	\]
	Since $P_j$ acts only in the $x$-variable, it commutes with each $T_J$.

	Fix \(h\in \mathcal S(\R^{1+2d})\). For the low-frequency bound, we use
	$P_j h=\check\eta_j *_x \widetilde P_j h$ together with \eqref{eq:max-x},
	\eqref{eq:max-bounds-J-1}, and \eqref{eq:max-bounds-tilde-1}. This yields
	\begin{align}
		|T_{\kpi_r}(P_j h)|+|T_{\kforc_r}(P_j h)|
		 & \lesssim
		M_{\rm kin}(M_x(|\widetilde P_j h|)), \label{eq:low-kin} \\
		|T_{\widetilde K_r}(P_j h)|
		 & \lesssim
		M_{{\rm kin},1}(M_x(|\widetilde P_j h|)). \label{eq:low-tilde-new}
	\end{align}

	For the high-frequency bound, we use \eqref{eq:Pj-grad} and integrate by parts
	in the kernel variable $y$. In the scalar case this gives
	\[
		T_{J_r}(P_j h)
		=
		-\,2^{-j}\sum_{\ell=1}^d
		T_{\partial_{y_\ell}J_r}\bigl(\Psi_{j,\ell} *_x \widetilde P_j h\bigr),
	\]
	and in the vector-valued case the same identity holds componentwise. Using
	\eqref{eq:max-x}, \eqref{eq:max-bounds-J-2}, and \eqref{eq:max-bounds-tilde-2}, we obtain
	\begin{align}
		|T_{\kpi_r}(P_j h)|+|T_{\kforc_r}(P_j h)|
		 & \lesssim
		2^{-j}r^{-3}\,M_{\rm kin}(M_x(|\widetilde P_j h|)), \label{eq:high-kin} \\
		|T_{\widetilde K_r}(P_j h)|
		 & \lesssim
		2^{-j}r^{-3}\,M_{{\rm kin},1}(M_x(|\widetilde P_j h|)). \label{eq:high-tilde-new}
	\end{align}

	Combining \eqref{eq:low-kin}--\eqref{eq:high-tilde-new} yields
	\begin{align*}
		2^{\frac j3}|T_{\kpi_r}(P_j h)| + 2^{\frac j3}|T_{\kforc_r}(P_j h)|
		 & \lesssim
		\min\Bigl\{2^{\frac j3},\,2^{-\frac{2j}{3}}r^{-3}\Bigr\}
		M_{\rm kin}(M_x(|\widetilde P_j h|)), \\
		2^{\frac j3}|T_{\widetilde K_r}(P_j h)|
		 & \lesssim
		\min\Bigl\{2^{\frac j3},\,2^{-\frac{2j}{3}}r^{-3}\Bigr\}
		M_{{\rm kin},1}(M_x(|\widetilde P_j h|)).
	\end{align*}
	Applying this with $h=S_0$, $h=S_1$, and $h=\nabla_v f$, we obtain
	\begin{align*}
		 & 2^{\frac j3}|P_j (f-T_{K_\tau}f)|                                       \\
		 & \lesssim
		\left(\int_0^\tau \min\{2^{\frac j3},2^{-\frac{2j}{3}}r^{-3}\}\dd r\right) \\
		 & \qquad \cdot
		\Bigl(
		M_{\rm kin}(M_x(|\widetilde P_j S_0|))
		+
		M_{\rm kin}(M_x(|\widetilde P_j \nabla_v f|))
		+
		M_{{\rm kin},1}(M_x(|\widetilde P_j S_1|))
		\Bigr).
	\end{align*}
	Since
	\[
		\int_0^\infty \min\{2^{\frac j3},2^{-\frac{2j}{3}}r^{-3}\}\dd r \lesssim 1
	\]
	uniformly in $j\in\Z$, it follows that
	\begin{align}\label{eq:pointwise-u-general}
		 & 2^{\frac j3}|P_j (f-T_{K_\tau}f)| \\
		 & \lesssim
		M_{\rm kin}(M_x(|\widetilde P_j S_0|))
		+
		M_{\rm kin}(M_x(|\widetilde P_j \nabla_v f|))
		+
		M_{{\rm kin},1}(M_x(|\widetilde P_j S_1|)). \nonumber
	\end{align}

	We now take the $\ell^2$ square function, then the $\L^p$ norm, and use
	\eqref{eq:LP-hom}, \eqref{eq:FS-kin}, \eqref{eq:FS-kin-frac}, \eqref{eq:FS-x},
	and \eqref{eq:LP-square}:
	\begin{align*}
		 & \|D_x^{\frac13}(f-T_{K_\tau}f)\|_{\L^p} \\
		 & \lesssim
		\left\|
		\left(\sum_j 2^{\frac{2j}{3}}|P_j (f-T_{K_\tau}f)|^2\right)^{1/2}
		\right\|_{\L^p}                            \\
		 & \lesssim
		\left\|
		\left(\sum_j |M_{\rm kin}(M_x(|\widetilde P_j S_0|))|^2\right)^{1/2}
		\right\|_{\L^p}
		\\
		 & \quad+
		\left\|
		\left(\sum_j |M_{\rm kin}(M_x(|\widetilde P_j \nabla_v f|))|^2\right)^{1/2}
		\right\|_{\L^p} 		+
		\left\|
		\left(\sum_j |M_{{\rm kin},1}(M_x(|\widetilde P_j S_1|))|^2\right)^{1/2}
		\right\|_{\L^p}                            \\
		 & \lesssim
		\left\|
		\left(\sum_j |\widetilde P_j S_0|^2\right)^{1/2}
		\right\|_{\L^p}
		+
		\left\|
		\left(\sum_j |\widetilde P_j \nabla_v f|^2\right)^{1/2}
		\right\|_{\L^p}
		+
		\left\|
		\left(\sum_j |\widetilde P_j S_1|^2\right)^{1/2}
		\right\|_{\L^q}                            \\
		 & \lesssim
		\|S_0\|_{\L^p}+\|\nabla_v f\|_{\L^p}+\|S_1\|_{\L^q}.
	\end{align*}
	Thus
	\begin{equation}\label{eq:u-tau-est-general}
		\|D_x^{\frac13}(f-T_{K_\tau}f)\|_{\L^p}
		\lesssim \|S_0\|_{\L^p}+\|\nabla_v f\|_{\L^p}+\|S_1\|_{\L^q},
	\end{equation}
	uniformly in $\tau>0$.

	\smallskip

	\noindent
	\textbf{Step 5: estimate of \(D_x^{1/3}T_{K_\tau}f\).}
	By Lemma~\ref{lem:commuteDx},
	\[
		D_x^{\frac13}T_{K_\tau}f = T_{D_x^{\frac13}K_\tau}f,
	\]
	where $D_x^{1/3}$ acts on the spatial kernel variable.
	Hence Young's inequality gives
	\[
		\|D_x^{\frac13}T_{K_\tau}f\|_{\L^p}
		\le
		\|D_x^{\frac13}K_\tau\|_{\L^1_{s,y,w}}\,
		\|f\|_{\L^p}.
	\]
	Since $0<\frac13<1$ and $K_\tau$ is smooth and compactly supported, the singular-integral formula for fractional derivatives \cite{MR3613319} yields
	\[
		D_y^{\frac13}K_\tau(s,y,w)
		=
		c_d\int_{\R^d}\frac{K_\tau(s,y,w)-K_\tau(s,y-h,w)}{|h|^{d+\frac13}}\,\dd h.
	\]
	Taking the $\L^1_{s,y,w}$ norm and using Fubini, we obtain
	\[
		\|D_x^{\frac13}K_\tau\|_{\L^1}
		\lesssim
		\int_{\R^d}\frac{\|\Delta_x^h K_\tau\|_{\L^1}}{|h|^{d+\frac13}}\,\dd h.
	\]
	By Corollary~\ref{cor:dxhK} we have $\|\Delta_x^h K_\tau\|_{\L^1}\lesssim \min\{1,|h|\tau^{-3}\}$.
	Therefore
	\begin{align*}
		\|D_y^{\frac13}K_\tau\|_{\L^1}
		 & \lesssim
		\int_{\R^d}\frac{\min\{1,|h|\tau^{-3}\}}{|h|^{d+\frac13}}\,\dd h \\
		 & \lesssim
		\int_0^{\tau^3}\tau^{-3}\rho^{-\frac13}\dd \rho
		+
		\int_{\tau^3}^\infty \rho^{-\frac43}\dd \rho
		\lesssim \tau^{-1}.
	\end{align*}
	Hence
	\[
		\|D_x^{\frac13}T_{K_\tau}f\|_{\L^p}
		\lesssim
		\tau^{-1}\|f\|_{\L^p}\to 0
		\qquad\text{as }\tau\to\infty.
	\]
	Together with \eqref{eq:u-tau-est-general}, this proves \eqref{eq:kintorLp-general}.
\end{proof}

We obtain the following multiplicative estimates.

\begin{coro}\label{cor:kintorLpMult}
	Under the assumptions of Theorem \ref{thm:kintorLpGeneral}, we have the multiplicative estimate
	\begin{equation}\label{eq:kintorLp-mult-sum}
		\|D_x^{\frac13}f\|_{\L^p}
		\lesssim
		\|\nabla_v f\|_{\L^p}^{\frac23}\|S_0\|_{\L^p}^{\frac13}
		+
		\|\nabla_v f\|_{\L^p}^{\frac{5d+1}{3(3d+1)}}
		\|S_1\|_{\L^q}^{\frac{2(2d+1)}{3(3d+1)}}.
	\end{equation}
\end{coro}

\begin{proof}
	Set
	\[
		A:=\|D_x^{\frac13}f\|_{\L^p},
		\qquad
		B:=\|\nabla_v f\|_{\L^p},
		\qquad
		C:=\|S_0\|_{\L^p},
		\qquad
		D:=\|S_1\|_{\L^q}.
	\]
	For $\lambda>0$ define
	\[
		f_\lambda(t,x,v):=f(\lambda t,\lambda x,v),
	\]
	and
	\[
		S_{0,\lambda}(t,x,v):=\lambda S_0(\lambda t,\lambda x,v),
		\qquad
		S_{1,\lambda}(t,x,v):=\lambda S_1(\lambda t,\lambda x,v).
	\]
	Then
	\[
		(\partial_t+v\cdot\nabla_x)f_\lambda
		=
		\nabla_v\!\cdot S_{0,\lambda}+S_{1,\lambda}.
	\]
	Applying Theorem \ref{thm:kintorLpGeneral} to $f_\lambda$ gives
	\[
		\|D_x^{\frac13}f_\lambda\|_{\L^p}
		\lesssim
		\|\nabla_v f_\lambda\|_{\L^p}
		+\|S_{0,\lambda}\|_{\L^p}
		+\|S_{1,\lambda}\|_{\L^q}.
	\]
	The scaling of the norms is
	\[
		\|D_x^{\frac13}f_\lambda\|_{\L^p}
		=
		\lambda^{\frac13-\frac{d+1}{p}}A,
		\qquad
		\|\nabla_v f_\lambda\|_{\L^p}
		=
		\lambda^{-\frac{d+1}{p}}B,
	\]
	\[
		\|S_{0,\lambda}\|_{\L^p}
		=
		\lambda^{1-\frac{d+1}{p}}C,
		\qquad
		\|S_{1,\lambda}\|_{\L^q}
		=
		\lambda^{1-\frac{d+1}{q}}D.
	\]
	Hence
	\[
		A
		\lesssim
		\lambda^{-\frac13}B
		+\lambda^{\frac23}C
		+\lambda^{1-\frac{d+1}{q}-\frac13+\frac{d+1}{p}}D
		=
		\lambda^{-\frac13}B+\lambda^{\frac23}C+\lambda^{\sigma_d}D
	\]
	for every $\lambda>0$. Taking the infimum yields
	\begin{equation}\label{eq:kintorLp-opt}
		\|D_x^{\frac13}f\|_{\L^p}
		\lesssim
		\inf_{\lambda>0}
		\Bigl(
		\lambda^{-\frac13}\|\nabla_v f\|_{\L^p}
		+\lambda^{\frac23}\|S_0\|_{\L^p}
		+\lambda^{\sigma_d}\|S_1\|_{\L^q}
		\Bigr),
	\end{equation}
	where
	\[
		\sigma_d
		:=
		\frac23-\frac{d+1}{\homd}
		=
		\frac{5d+1}{6(2d+1)}.
	\]

	If $B=0$, then \eqref{eq:kintorLp-opt} gives
	\[
		A\lesssim \lambda^{\frac23}C+\lambda^{\sigma_d}D
	\]
	for every $\lambda>0$. Since $\sigma_d>0$, letting $\lambda\downarrow0$ yields
	$A=0$, so \eqref{eq:kintorLp-mult-sum} follows trivially. Hence we may assume
	in the remainder that $B>0$.

	If $C=D=0$, then \eqref{eq:kintorLp-opt} gives
	\[
		A\lesssim \lambda^{-\frac13}B
	\]
	for every $\lambda>0$, so letting $\lambda\to\infty$ yields $A=0$. Thus in what
	follows we may assume that at least one of $C,D$ is nonzero.

	If $D=0$, choose $\lambda=\frac{B}{C}$. Then the first two terms in
	\eqref{eq:kintorLp-opt} are equal, and therefore
	\[
		A\lesssim B^{\frac23}C^{\frac13}.
	\]
	If $C=0$, choose
	\[
		\lambda=\Bigl(\frac{B}{D}\Bigr)^{\frac{1}{\sigma_d+\frac13}}.
	\]
	Then the first and third terms in \eqref{eq:kintorLp-opt} are equal, and hence
	\[
		A
		\lesssim
		B^{\frac{\sigma_d}{\sigma_d+\frac13}}
		D^{\frac{\frac13}{\sigma_d+\frac13}}
		=
		B^{\frac{5d+1}{3(3d+1)}}D^{\frac{2(2d+1)}{3(3d+1)}}.
	\]

	If $C,D>0$, set
	\[
		\lambda_0:=\frac{B}{C},
		\qquad
		\lambda_1:=\Bigl(\frac{B}{D}\Bigr)^{\frac{1}{\sigma_d+\frac13}}.
	\]
	At $\lambda=\lambda_0$ the first two terms in \eqref{eq:kintorLp-opt} are comparable:
	\[
		\lambda_0^{-\frac13}B=\lambda_0^{\frac23}C=B^{\frac23}C^{\frac13}.
	\]
	At $\lambda=\lambda_1$ the first and third terms are comparable:
	\[
		\lambda_1^{-\frac13}B=\lambda_1^{\sigma_d}D
		=
		B^{\frac{\sigma_d}{\sigma_d+\frac13}}
		D^{\frac{\frac13}{\sigma_d+\frac13}}
		=
		B^{\frac{5d+1}{3(3d+1)}}D^{\frac{2(2d+1)}{3(3d+1)}}.
	\]

	If $\lambda_0\le \lambda_1$, then
	\[
		\lambda_0^{\sigma_d}D
		\le
		\lambda_1^{\sigma_d}D
		=
		\lambda_1^{-\frac13}B
		\le
		\lambda_0^{-\frac13}B,
	\]
	so choosing $\lambda=\lambda_0$ in \eqref{eq:kintorLp-opt} gives
	\[
		A\lesssim B^{\frac23}C^{\frac13}.
	\]
	If $\lambda_1\le \lambda_0$, then
	\[
		\lambda_1^{\frac23}C
		\le
		\lambda_0^{\frac23}C
		=
		\lambda_0^{-\frac13}B
		\le
		\lambda_1^{-\frac13}B,
	\]
	so choosing $\lambda=\lambda_1$ in \eqref{eq:kintorLp-opt} gives
	\[
		A
		\lesssim
		B^{\frac{5d+1}{3(3d+1)}}D^{\frac{2(2d+1)}{3(3d+1)}}.
	\]
	This proves \eqref{eq:kintorLp-mult-sum}.
\end{proof}

\bibliographystyle{plain}

\end{document}